\documentclass[12pt]{amsart}
\usepackage{txfonts}
\usepackage{amscd,amssymb,mathabx,subfigure}
\usepackage[arrow,matrix,graph,frame,poly,arc,tips]{xy}
\usepackage{graphicx,calrsfs}
\usepackage[colorlinks,plainpages,backref,urlcolor=blue]{hyperref}
\usepackage{tikz}
\usetikzlibrary{decorations.markings}
\usetikzlibrary{shapes}
\usetikzlibrary{backgrounds}
\usetikzlibrary{calc}
\usepackage{mdwlist}
\usepackage{multirow}
\usepackage{enumerate}
\usepackage[shortlabels]{enumitem}
\usepackage{verbatim}
\usepackage[abs]{overpic}

\title[Two--chains and square roots of Thompson's group $F$]{Two--chains and square roots of Thompson's group $F$}

\author{Thomas Koberda}
\address{Department of Mathematics, University of Virginia, 
Charlottesville, VA 22904-4137, USA}
\email{thomas.koberda@gmail.com}
\urladdr{\href{http://faculty.virginia.edu/Koberda/}%
{http://faculty.virginia.edu/Koberda/}}

\author[Y. Lodha]{Yash Lodha}
\address{EPFL SB MATH EGG, Station 8, MA C3 584 (B\^atiment MA), Station 8,  CH-1015, Lausanne, Switzerland}
\email{yash.lodha@epfl.ch}
\urladdr{\href{https://people.epfl.ch/yash.lodha}{https://people.epfl.ch/yash.lodha}}

\usepackage{enumerate}
\makeatletter
\let\@@enum@org\@@enum@
\def\@@enum@[#1]{\@@enum@org[\normalfont #1]}
\makeatother

\newtheorem{thm}{Theorem}[section]
\newtheorem{lem}[thm]{Lemma}

\newtheorem{cor}[thm]{Corollary}
\newtheorem{prop}[thm]{Proposition}

\theoremstyle{definition}

\newcommand\R{\mathbb{R}}

\newcommand\Z{\mathbb{Z}}

\newcommand\bR{\mathbb{R}}
\newcommand\bN{\mathbb{N}}

\newcommand\form[1]{\langle #1\rangle}

\DeclareMathOperator\Diff{Diff}

\newcommand\supp{\operatorname{supp}}

\newcommand\Homeo{\operatorname{Homeo}}

\newcommand\yh{\widehat}

\begin{document}

\date{\today}



\begin{abstract}
We study two--generated subgroups $\langle f,g\rangle<\Homeo^+(I)$ such that $\langle f^2,g^2\rangle$ is isomorphic to Thompson's group $F$, and such that the supports of $f$ and $g$ form a chain of two intervals. We show that this class contains 
uncountably many isomorphism types. 
These include examples with nonabelian free subgroups, examples which do not admit faithful actions by $C^2$ diffeomorphisms on $1$--manifolds, examples which do not admit faithful actions by $PL$ homeomorphisms on an interval, and examples which are not finitely presented. We thus answer questions due to M. Brin. We also show that many relatively uncomplicated groups of homeomorphisms can have very complicated square roots, thus establishing the behavior of square roots of $F$ as part of a general phenomenon among subgroups of $\Homeo^+(I)$.
\end{abstract}
\maketitle

\section{Introduction}\label{sec:intro}

Thompson's group $F$ is a remarkable group of piecewise linear (abbreviated $PL$) homeomorphisms of the interval $I=[0,1]$ that occurs naturally and abundantly as a group of homeomorphisms of the real line, and that has been extensively studied since the 1970s.
The group $F$ has been shown to satisfy various exotic properties,
and has been proposed as a counterexample to well--known conjectures in group theory \cite{BrownGeoghegan,CFP1996,BieriStrebel16}.
Among the most well--known facts about Thompson's group $F$ are the following:

\begin{thm}[Brin--Squier,~\cite{BS1985}]\label{thm:brinsquier}
The group $F$ satisfies no law and contains no nonabelian free subgroups.
\end{thm}

\begin{thm}[Ghys--Sergiescu,~\cite{GS1987}]\label{thm:ghyssergiescu}
The group $F$ admits a faithful action by $C^{\infty}$ diffeomorphisms of the circle.
\end{thm}

\begin{thm}[Thompson, see~\cite{CFP96}]\label{thm:fp}
The group $F$ is finitely presented.
\end{thm}

\begin{thm}[See~\cite{CFP96,BurilloBook,Higman,HigmanBook,Brown1985}]\label{thm:simple}
The commutator subgroup of $F$ is an infinite simple group.
\end{thm}

In this article, we study a certain class of groups which we call \emph{square roots of Thompson's group $F$}. 
These are two--generated subgroups $\langle f,g\rangle <\Homeo^+(I)$ of the group of orientation-preserving homeomorphisms of the interval, which satisfy 
\[\langle f^2,g^2\rangle\cong \form{A,B \mid [A, (AB)^{-k}B(AB)^k]\text{ for }k\in\{1,2\}}
\cong F,\] and for which the supports $\supp f$ and $\supp g$ of $f$ and $g$ respectively form a \emph{two--chain} of intervals. That is, $\supp f$ and $\supp g$ are both open intervals, and the intersection $\supp f\cap \supp g$ is a proper subinterval of both $\supp f$ and $\supp g$.

Among other things, we demonstrate that (the second part of) Theorem~\ref{thm:brinsquier}, Theorem~\ref{thm:ghyssergiescu}, and Theorem \ref{thm:fp} all fail for square roots of $F$.
In particular, we show that there are square roots of $F$ which contain nonabelian free subgroups, that there are square roots of $F$ which do not admit faithful actions by $C^2$ diffeomorphisms on the interval, circle, or real line, and that there are uncountably many isomorphism types of square roots of $F$.

%%\cong \form{A,B \mid [AB^{-1}, A^{-k}BA^k]\text{ for }k\in\{1,2\}}.\]

\subsection{Main results}
We denote the set of isomorphism classes of square roots of $F$ by $\mathcal{S}$. The goal of this paper is to produce interesting elements of $\mathcal{S}$. Note that $\mathcal{S}$ contains $F$ for example, since squaring the generators in the standard presentation for $F$ as given in the previous subsection results in a group isomorphic to $F$.

In this article we use two different finite presentations of the group $F$.
The first presentation, which was mentioned in the previous section is:
\[\form{A,B \mid [A, (AB)^{-1}B(AB)], [A, (AB)^{-2}B(AB)^2]}
\cong F.\]
The second presentation is obtained by performing a Tietze transformation to produce generators $a=AB,\, b=B$, and is given by:

\[\form{a,b \mid [ab^{-1}, a^{-1}ba], [ab^{-1}, a^{-2}ba^2]}
\cong F.\]

Next, we describe a certain subgroup $P$ of $F$, which will be needed to state and prove our results.
We fix two copies of $F$: \[F_1=\form{p_1,p_2 \mid [p_1p_2^{-1}, p_1^{-1}p_2p_1],[p_1p_2^{-1}, p_1^{-2}p_2p_1^2]},\] 
\[F_2=\form{q_1,q_2 \mid [q_1q_2^{-1}, q_1^{-1}q_2q_1],[q_1q_2^{-1}, q_1^{-2}q_2q_1^2]}.\]
We will write $P$ for the subgroup of $F_1\times F_2$ generated by $(p_1,q_2)$ and $(p_2,q_1)$.
Note that the group $P$ is isomorphic to a subgroup of $F$ which itself contains an isomorphic copy of $F$ as a subgroup.
The fact that $P$ is isomorphic to a subgroup of $F$ is an elementary exercise that we leave to the reader, and that $F$ is isomorphic to a
subgroup of $P$ is a direct consequence of Brin's Ubiquity Theorem (see~\cite{BrinJLMS99}). 

We denote the free group on two generators by ${\bf F_2}$, and we call a group $H=\langle h_1,h_2\rangle$ a \emph{marked extension} of $P$ if there exists a surjective homomorphism $H\to P$, where \[h_1\mapsto (p_1,q_2),\,h_2\mapsto (p_2,q_1).\] 
Even though the map $H\to P$ may be suppressed from the notation, we always think of a marked extension of $P$ as equipped with such a homomorphism. 

A (countable) group is \emph{left orderable} if it admits a left invariant total ordering,
or equivalently if it admits a faithful action by orientation preserving homeomorphisms of the real line
(see Proposition 1.1.8 of~\cite{DeroinNavasRivas} or Theorem 2.2.19 of~\cite{Navas2011}). 
Our main result is the following:

\begin{thm}\label{thm:main}
Let $H$ be a marked, left orderable extension of $P$. 
Then there exists a square root of Thompson's group $G\in\mathcal{S}$ such that $H<G$.
\end{thm}

Since the free group $\bf F_2$ is left orderable and is naturally a marked extension of $P$, we immediately obtain the following:

\begin{cor}\label{cor:free}
There exists a square root $G\in\mathcal{S}$ such that ${\bf F_2}<G$.
\end{cor}

We will show that square roots of $F$ can contain torsion--free nilpotent groups of arbitrary nilpotence degree. 
As a consequence of Theorem \ref{thm:main} and the Plante--Thurston Theorem~\cite{PT1976}, we have the following:

\begin{cor}\label{cor:smooth}
There exists a square root $G\in\mathcal{S}$ such that $G$ does not admit a faithful action by $C^2$ diffeomorphisms on a compact one--manifold or on the real line.
\end{cor}

Corollary \ref{cor:smooth} gives an example of a subgroup $\langle f,g\rangle<\Homeo^+(I)$ which admits no faithful $C^2$ action on the interval, the circle, or the real line, but where $\langle f^2,g^2\rangle$ admits a faithful $C^{\infty}$ action on every one--manifold (cf.~\cite{GS1987,KKL16}).

Nonabelian nilpotent groups cannot act by piecewise--linear homeomorphisms on $I$ or on $S^1$:

\begin{cor}\label{cor:pl}
There exists a square root $G\in\mathcal{S}$ such that $G$ does not admit a faithful action by $PL$ homeomorphisms of a compact one--manifold.
\end{cor}

Corollary~\ref{cor:pl} stands in contrast to the standard definition of $F$, which is as a group of $PL$ homeomorphisms of the interval. Corollary~\ref{cor:pl} answers a question due to M. Brin~\cite{BrinPersonal}.

In order to show that square roots of $F$ may not be finitely presented, we prove the following result which is similar in spirit to some of the methods in~\cite{KKL16}:

\begin{thm}\label{thm:uncountable}
The class $\mathcal{S}$ contains uncountably many distinct isomorphism types.
\end{thm}

Since there are only countably many isomorphism types of finitely presented groups, we immediately obtain the following:

\begin{cor}\label{cor:fp}
There exists an element $G\in\mathcal{S}$ which admits no finite presentation.
\end{cor}

Analogues of Theorem~\ref{thm:simple} for square roots of $F$ are not the primary topic of this paper, though we can give the following statement which follows immediately from the discussion of commutator subgroups of chain groups in~\cite{KKL16}. Recall that the action of a group on a topological space is \emph{minimal} if every point has a dense orbit:

\begin{prop}\label{prop:simple}
Let $G\in\mathcal{S}$ act minimally on its support. Then the commutator subgroup $[G,G]$ is simple.
\end{prop}

It follows from Proposition~\ref{prop:simple} that if $G=\langle f,g\rangle$ and that if $\langle f^2,g^2\rangle$ generate a copy of Thompson's group $F$ which acts minimally on the interior of $I$, then the commutator subgroup of $G$ is simple.

\subsection{Square roots of other groups}\label{subsec:other}

Essential in the discussion of square roots of $F$ in this paper is the \emph{dynamical realization} of $F$ on a two--chain of intervals, which is a dynamical setup in which $F$ occurs naturally (see Subsection~\ref{subsec:2prechain}, cf. Proposition 1.1.8 of~\cite{DeroinNavasRivas}). If one abandons the dynamical framework of chains of intervals, the group theoretic diversity phenomena witnessed by Theorems~\ref{thm:main} and~\ref{thm:uncountable} become so common as to be a general feature of homeomorphism groups.

To be precise, let $H=\langle h_1,\ldots,h_n\rangle<\Homeo^+(I)$ be a finitely generated subgroup. An $n$--generated subgroup $G=\langle g_1,\ldots,g_n\rangle<\Homeo^+(I)$ is called a \emph{square root of $H$} if \[H\cong\langle g_1^2,\ldots,g_n^2\rangle.\]
We note that the definition of a square root of $H$ depends implicitly on a choice of generators for $H$, and is therefore really a square root of
a marked group.

If $H=\langle h_1,\ldots,h_n\rangle$ is a generating set for a group $H$, we will define the \emph{skew subdirect product} of $H$ to be the subgroup of $H\times H$ generated by $\{(h_i,h_i^{-1})\}_{i=1}^n$, and we will denote this group by $\yh{H}$.

\begin{thm}\label{thm:Z^n}
Let $\Z=\langle t_1,\ldots,t_{n+1}\mid t_1=\cdots=t_{n+1}\rangle$, and let $H<\Homeo^+(I)$ be an $n$--generated group. Then there exists a square root $G$ of $\Z$ such that $\yh H<G$.
\end{thm}

\begin{cor}\label{cor:uncgeneral}
There exist uncountably many isomorphism types of three--generated subgroup of $\Homeo^+(I)$ such that the squares of the generators generate a cyclic group. Moreover, there exists a three--generated subgroup of $\Homeo^+(I)$ such that the squares of the generators generate a cyclic group and which contains a nonabelian free group.
\end{cor}

\begin{thm}\label{thm:lamplighter}
Let $L=\Z\wr\Z$ be the lamplighter group, equipped with standard cyclic generators of the two factors of the wreath product.
Then $L$ has uncountably many isomorphism types of (marked) square roots.
\end{thm}

In Subsection~\ref{subsec:generalroot}, we will define the notion of a formal square root of a finitely generated group. We will show that formal square roots of left orderable groups are again left orderable, and generally contain nonabelian free groups.

\subsection{Notes and references}
\subsubsection{Remarks on context}
The bulk of the present work could just as well be a discussion of the very general setup of two--generated subgroups of $\Homeo^+(I)$ whose generators are supported on intervals $J_1$ and $J_2$, which in turn form a chain. It is well--known that under suitable dynamical hypotheses (cf. Subsection~\ref{sssec:relation} below), the resulting subgroup is isomorphic to $F$. The class $\mathcal{S}$ of square roots of $F$ is merely the first instance of interesting algebraic behavior for such homeomorphism groups which does not follow from the properties of Thompson's group $F$. In particular, the results of this article apply to higher roots of $F$ beyond the square root.

\subsubsection{Relation to other authors' work}\label{sssec:relation}
To the authors' knowledge, it was M. Brin~\cite{BrinPersonal} who first asked what sorts of groups can occur as square roots of $F$, and in particular if square roots of $F$ can contain nonabelian free groups, whether they can fail to be finitely presented, and whether they can fail to act by $PL$ homeomorphisms on the interval. The main results of this paper form a natural complement to the joint work of the authors with S. Kim in~\cite{KKL16}. In that paper, Kim and the authors introduced the notions of a \emph{prechain group} and of a \emph{chain group}. In the terminology of~\cite{KKL16}, square roots of $F$ form a restricted subclass of $2$--prechain groups, namely those which square to become $2$--chain groups. The class of $2$--chain groups in turn consists of just one isomorphism type (i.e. Thompson's group $F$). Chain groups with ``fast" dynamics also fall into very few isomorphism types (namely the Higman--Thompson groups $\{F_n\}_{n\geq 2}$, and their subgroup structure has been studied independently by Bleak--Brin--Kassabov--Moore--Zarmesky~\cite{BBKMZ16} (cf.~\cite{Brown1985}). For generalities on Thompson's group $F$, the reader is directed to the classical Cannon--Floyd--Parry notes~\cite{CFP96}, as well as Burillo's book~\cite{BurilloBook}.

\subsubsection{Bi-orderability}
We briefly remark that many of the groups we construct in this paper, though they are manifestly orderable, will fail to be bi-orderable. Indeed,
bi-orderable groups are known to have the \emph{unique root property}. That is, if $f^n=g^n$ for some elements $f$ and $g$ in a biorderable
group for some $n\neq 0$, then $f=g$ (see Section 1.4.2 of~\cite{DeroinNavasRivas}). One of the themes of this paper is the non-uniqueness
of roots of homeomorphisms. Thus, the moment a given element in a left orderable group has two distinct square roots, the group cannot be
bi-orderable. See for instance Corollary~\ref{cor:uncgeneral}.

\section{Square roots of $F$}
In this section we establish the main result, after gathering some relevant preliminary facts and terminology.

\subsection{$2$--prechain groups and variations thereupon}\label{subsec:2prechain}

Let $\mathcal{J}=\{J_1,J_2\}$ be two nonempty open subintervals of $\bR$. 
We call $\mathcal{J}$ a \emph{chain of intervals} if $J_1\cap J_2$ is a proper nonempty subinterval of $J_1$ and of $J_{2}$. See Figure~\ref{f:coint}.
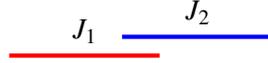
\begin{figure}[h!]
\begin{tikzpicture}[ultra thick,scale=.5]
\draw [red] (-5,0) -- (-1,0);
\draw (-3,0) node [above] {\small $J_1$}; 
\draw [blue] (-2,.5) -- (2,.5);
\draw (0,.5) node [above] {\small $J_2$}; 
%\draw [teal] (1,0) -- (5,0);
%\draw (3,0) node [above] {\small $J_3$}; 
\end{tikzpicture}
\caption{A chain of two intervals.}
\label{f:coint}
\end{figure}

If $f\in\Homeo^+(\R)$, we write $\supp f=\{x\in\R\mid f(x)\neq x\}$. Let $f$ and $g$ satisfy $\supp f= J_1$ and $\supp g=J_2$. In the terminology of~\cite{KKL16}, the group $\langle f,g\rangle$ is a $2$--prechain group. Note that, up to replacing $f$ and $g$ by their inverses, we may assume $f(x),g(x)\geq x$ for $x\in\R$.

Writing $J_1=(a,c)$ and $J_2=(b,d)$ with $a<b<c<d$, we have the following basic dynamical stability result, a proof of which can be found as a special case of Lemma 3.1 of~\cite{KKL16}:

\begin{lem}\label{lem:dyn criterion}
Suppose $g\circ f(b)\geq c$. Then $\langle f,g\rangle\cong F$.
\end{lem}

Under the dynamical hypotheses of Lemma \ref{lem:dyn criterion}, the group $\langle f,g\rangle$ is a chain group.
There is another configuration of intervals and homeomorphisms closely related to chain groups, which naturally gives rise to $F$, which we will need in the sequel, and which we will describe in the next subsection.

\subsection{Nested generators for $F$}
A natural generating set for $F$ emerges as homeomorphisms supported on a nested pair of intervals,
satisfying elementary dynamical conditions.
This shall be useful in our construction to follow.

%%Two nonempty compact intervals $\{J_1,J_2\}$ of $\bR$ are said to be \emph{nested}
%if $J_2\subset J_1$ and either one of the following holds:
%$$inf(J_2)= inf(J_1)\qquad sup(J_2)<sup(J_1)$$
%or 
%$$sup(J_2)=sup(J_1)\qquad inf(J_2)> inf(J_1)$$
%In the former case we say that $J_1,J_2$ are \emph{left nested}, and in the latter case we say that they are \emph{right nested}.
%Given homeomorphisms supported on nested intervals, elementary dynamical conditions guarantee that they generate $F$. 

\begin{lem}\label{nestedL}
Let $[a,b_1]$ and $ [a,b_2]$ be compact intervals in $\mathbb{R}$ such that $b_1<b_2$.
Let $f,g$ be homeomorphisms satisfying:
\begin{enumerate}
\item The supports of $g$ and $f$ are contained in $[a,b_1]$ and $ [a,b_2]$ respectively. 
\item $f$ is a decreasing map on $(a,b_2)$.
\item $f,g$ agree on the interval $[a, f(b_1)]$.
\end{enumerate}
Then $\langle f,g \rangle\cong F$.
\end{lem}

\begin{lem}\label{nestedR}
Let $[a_1,b]$ and $ [a_2,b]$ be compact intervals in $\mathbb{R}$ such that $a_1<a_2$.
Let $f,g$ be homeomorphisms satisfying:
\begin{enumerate}
\item The supports of $f$ and $g$ are contained in $[a_1,b]$ and $[a_2,b]$ respectively.
\item $f$ is an increasing map on $(a_1,b)$.
\item $f,g$ agree on the interval $[f(a_2), b]$.
\end{enumerate}
Then $\langle f,g \rangle\cong F$.
\end{lem}

\begin{proof}[Proofs of Lemmas~\ref{nestedL} and~\ref{nestedR}] The proofs of both lemmas above follow from checking that the homeomorphisms $f$ and $g$ in each lemma
satisfy the relations $$[fg^{-1},f^{-1}gf]=1\qquad [fg^{-1}, f^{-2}g f^2]=1$$
Since $f$ and $g$ do not commute, and since every proper quotient of $F$ is abelian (see Theorem 4.3 of~\cite{CFP96}), they generate a group isomorphic to $F$.  
\end{proof}

\subsection{Orderable extensions of $P$}\label{sec:extension}

We will use the following standard facts from the theory of orderable groups:

\begin{lem}[See Remark 2.1.5 of~\cite{DeroinNavasRivas}]\label{lem:extension}
Let $1\to K\to G\to Q\to 1$ be an exact sequence of groups, and suppose that $Q$ and $K$ are left orderable. 
Then $G$ admits a left ordering which agrees with any prescribed ordering on $K$. Moreover, any countable, left orderable group can be embedded in $\Homeo^+(I)$.
\end{lem}

The second claim of the lemma is implied by the fact that $\R\cong (0,1)$.
The following lemma is obvious, after the observation that $F\times F$ is left orderable, applying the Brin--Squier Theorem~\cite{BS1985}, and Brin's Ubiquity Theorem~\cite{BrinJLMS99}:

\begin{lem}\label{lem:P}
The group $P$ is a two--generated sub-direct product of $F\times F$. It is a left orderable group which contains no free subgroups.
\end{lem}

Let $\mathcal{R}<{\bf F_2}=\langle A,B\rangle$ be such that ${\bf F_2}/\mathcal{R}\cong P$. Here, the generators $A$ and $B$ of ${\bf F_2}$
get sent to the generators $(p_1,q_2)$ and $(p_2,q_1)$ respectively.
Note that since $P<F\times F$, we have $\mathcal{R}\neq 1$. Let $\mathcal{R}_k$ denote the $k^{th}$ term of the derived series of $\mathcal{R}$ and let $\mathcal{R}^k$ denote the $k^{th}$ term of the lower central series of $\mathcal{R}$, with the convention $\mathcal{R}_1=\mathcal{R}^1=\mathcal{R}$.

\begin{lem}\label{lem:solvable}
For each $k\geq 1$, the groups \[S_k=\langle A,B\mid \mathcal{R}_k\rangle\,\,\,\,\textrm{ and } \,\,\,\, N_k=\langle A,B\mid \mathcal{R}^k\rangle\] are marked, left orderable extensions of $P$.
\end{lem}
\begin{proof}
It is clear that for each $k$, the groups $S_k$ and $N_k$ are quotients of the free group $\bf F_2$ via the canonical map. 
Since $\mathcal{R}_k,\mathcal{R}^k\subset\mathcal{R}$, we have that $S_k$ and $N_k$ both surject to $P$ simply by imposing the relations in $\mathcal{R}$. It therefore suffices to show that $S_k$ and $N_k$ are both left orderable, which since $P$ is left orderable, reduces to showing that $\mathcal{R}/\mathcal{R}_k$ and $\mathcal{R}/\mathcal{R}^k$ are left orderable by Lemma \ref{lem:extension}.

Since $\mathcal{R}$ is an infinitely generated free group, these quotients are merely the universal $k$--step solvable and nilpotent quotients of the infinitely generated free group. We proceed by induction on $k$. The case $k=1$ is trivial, and in the case $k=2$, we obtain the group $\Z^{\infty}$ which is easily seen to be left orderable. By induction, $\mathcal{R}/\mathcal{R}_k$ (resp. $\mathcal{R}/\mathcal{R}^k$) is left orderable, and $\mathcal{R}_k/\mathcal{R}_{k+1}$ (resp. $\mathcal{R}^k/\mathcal{R}^{k+1}$) is again isomorphic to $\Z^{\infty}$, so the conclusion follows by applying Lemma \ref{lem:extension} again.
\end{proof}

\subsection{Building square roots of $F$}\label{subsec:building}

In this section we provide a recipe that produces a square root of $F$ that contains a given group $H$ as a subgroup, provided $H$ is an orderable marked extension of $P$.

{\bf Step 1}: Partition $[1,2)$ into left closed, right open intervals $\{J_1,\ldots,J_{16}\}$ so that $J_i$ 
occurs to the left of $J_j$ in $\mathbb{R}$
whenever $i<j$.
Moreover, we require that these intervals are of the same length.
For ease of notation, we denote by $J_X$ for some $X\subset \{1,\ldots,16\}$ the union $\bigcup_{i\in X} J_i$.
For example, $$J_{\{1,\ldots,4\}}=J_1\cup J_2\cup J_3\cup J_4.$$

{\bf Step 2}: Construct homeomorphisms $f$ and $g$ of the real line that satisfy the following:
\begin{enumerate}
\item $f$ and $g$ are increasing maps on $(0,2)$ and on $(1,3)$, respectively, and equal the identity outside these respective intervals.
\item $f$ maps $J_i$ isometrically onto $J_{i+4}$ for $1\leq i\leq 11$.
\item $g$ maps $J_i$ isometrically onto $J_{i+4}$ for $2\leq i\leq 12$.
\item The map $gf^{-1}$ has two components of support, which are $$[0,1]\cup J_1,\qquad J_{\{12,\ldots,16\}}\cup [2,3].$$
\end{enumerate}

It is elementary to construct homeomorphisms $f$ and $g$ that satisfy (1)--(3) above.
If $f$ and $g$ satisfy (1)--(3), then it holds that $gf^{-1}$ is the identity on $J_{\{2,\ldots,11\}}$.
Hence the support of $gf^{-1}$ is contained in 
$$[0,1]\cup J_1\bigcup J_{\{12,\ldots,16\}}\cup [2,3].$$
To ensure that the components of support of $gf^{-1}$ are precisely as stated in $(4)$,
we choose $g$ such that it is sufficiently slow on $J_1$,
and $f$ so that it is sufficiently slow on $J_{\{12,\ldots,16\}}$.
Note that $gf^{-1}$ is decreasing on the interior of $[0,1]\cup J_1$ and increasing on the interior of $J_{\{12,\ldots,16\}}\cup [2,3]$.

{\bf Step 3}: Let $H=\langle h_1,h_2\rangle$ be a marked, left orderable extension of $P$.
We identify the elements $h_1$ and $h_2$ with their dynamical realizations, both supported on the interval $J_6$. Here, by
\emph{dynamical realization} of a countable left orderable group $H$, we mean an embedding of into $\Homeo^+(I)$ (see Proposition
1.1.8 of~\cite{DeroinNavasRivas} or Theorem 2.2.19 of~\cite{Navas2011}).
Define a map $h_3$ as: $$h_3=g^{-1}h_2g=f^{-1}h_2f.$$
By definition, $h_3$ is supported on the interval $J_{10}$.
Finally, we define homeomorphisms $$\lambda_1= h_1^{-1} h_3^{-1} f,\qquad \lambda_2=g.$$

Our goal for the rest of this section will be to demonstrate the following:

\begin{prop}\label{main}
The group $\langle \lambda_1,\lambda_2\rangle$ is a marked square root of $F$ which contains $H$ as a subgroup. 
\end{prop}

The group $\langle \lambda_1,\lambda_2\rangle$ is manifestly orderable, since it is presented as a group of orientation preserving homeomorphisms of the interval. It is clear by our construction that $\lambda_1^2$ and $ \lambda_2^2$ satisfy the dynamical condition of Lemma \ref{lem:dyn criterion},
and hence generate a copy of $F$.
So it suffices to show that $H<\langle \lambda_1,\lambda_2\rangle$.

\begin{prop}\label{mainsub}
The elements $\lambda_2\lambda_1^{-1}$ and $\lambda_1^{-1}\lambda_2$ generate an isomorphic copy of $H$.
\end{prop}

\begin{proof}
The element $\lambda_2\lambda_1^{-1}$ has four components of support:
$$[0,1]\cup J_1,\qquad J_6,\qquad J_{10},\qquad J_{\{12,\ldots,16\}}\cup [2,3].$$
Note that:
$$\lambda_2\lambda_1^{-1}\restriction J_6= h_1,\qquad \lambda_2\lambda_1^{-1}\restriction J_{10}= h_3=f^{-1}h_2f.$$

We denote by $p_2$ the following restriction: 
$$\lambda_2\lambda_1^{-1}\restriction [0,1]\cup J_1= gf^{-1}\restriction [0,1]\cup J_1.$$
We denote by $q_1$ the following restriction:
$$\lambda_2\lambda_1^{-1}\restriction J_{\{12,\ldots,16\}}\cup [2,3]=gf^{-1}\restriction J_{\{12,\ldots,16\}}\cup [2,3].$$

The element $\lambda_1^{-1}\lambda_2$ has four components of support:
$$[0,1]\cup J_{\{1,\ldots,5\}},\qquad J_{10},\qquad J_{14},\qquad J_{16}\cup [2,3].$$
Note that 
$$\lambda_1^{-1}\lambda_2\restriction J_{10}= f^{-1}h_1 f.$$

Denote by $p_1$ the following restriction: 
$$\lambda_1^{-1}\lambda_2\restriction [0,1]\cup J_{\{1,\ldots,5\}}=f^{-1}g\restriction [0,1]\cup J_{\{1,\ldots,5\}}.$$
Denote by $q_2$ the following restriction:
$$\lambda_1^{-1}\lambda_2\restriction J_{\{14,\ldots,16\}}\cup [2,3].$$

First observe that the restrictions on $J_{10}$ are:
$$\lambda_1^{-1}\lambda_2\restriction J_{10}=f^{-1}h_1f\restriction J_{10},\qquad   
\lambda_2\lambda_1^{-1}\restriction J_{10}=h_3\restriction J_{10}=f^{-1} h_2 f\restriction J_{10}.$$
It follows that this restriction to $J_{10}$ corresponds to the isomorphism $$H\to \langle\lambda_1^{-1}\lambda_2\restriction J_{10}, \lambda_2\lambda_1^{-1}\restriction J_{10} \rangle,$$ 
given by $$h_2\mapsto \lambda_2\lambda_1^{-1}\restriction J_{10},\qquad h_1\mapsto \lambda_1^{-1}\lambda_2\restriction J_{10},$$
since these restrictions generate a dynamical realization of $H$ on $J_{10}$.

Next observe that 
$$\lambda_2\lambda_1^{-1}\restriction J_6=h_1\restriction J_6,\qquad \lambda_1^{-1}\lambda_2\restriction J_6=id\restriction J_6.$$
Since $H$ is a marked extension of $P$, every relation in $H$ is necessarily a product of commutators.
It follows that the abelianization of $H$ is $\mathbb{Z}^2$.
We have then that this restriction to $J_6$ corresponds to the quotient $$H\to \langle\lambda_1^{-1}\lambda_2\restriction J_6, \lambda_2\lambda_1^{-1}\restriction J_6 \rangle,$$ 
given by $$h_2\mapsto \lambda_2\lambda_1^{-1}\restriction J_6,\qquad h_1\mapsto \lambda_1^{-1}\lambda_2\restriction J_6.$$
which is a homomorphism whose kernel is the normal closure of $h_1$ in $H$.

Next, we observe that by construction, 
the maps $p_1,p_2$ and $q_1,q_2$ satisfy the dynamical conditions described in Lemmas \ref{nestedL} and \ref{nestedR} respectively.
Define $$j_1=\sup(J_1),\qquad j_2=\inf(J_{14}).$$
By construction, we have 
$$p_1(j_1)=\lambda_1^{-1}\lambda_2(j_1)= f^{-1}g(j_1)=g (f^{-1}(j_1))< 1,$$
and $$p_1\restriction [0,1]=\lambda_1^{-1}\lambda_2\restriction [0,1]= f^{-1}\restriction [0,1]= \lambda_2\lambda_1^{-1}\restriction [0,1]=p_2\restriction [0,1].$$
It follows that:
$$\langle p_1,p_2\rangle \cong \langle p_1,p_2\mid [p_1p_2^{-1},p_1^{-1}p_2p_1], [p_1p_2^{-1},p_1^{-2}p_2p_1^2]\rangle\cong F.$$

Next, observe that by construction we have 
$$q_1(j_2)=\lambda_2\lambda_1^{-1}(j_2)= gf^{-1}(j_2)=f^{-1}(g (j_2))> 2,$$
and $$q_2\restriction [2,3]=\lambda_1^{-1}\lambda_2\restriction [2,3]= g\restriction [2,3]= \lambda_2\lambda_1^{-1}\restriction [2,3]=q_1\restriction [2,3].$$
It follows that:
$$ \langle q_1,q_2\rangle\cong \langle q_1,q_2\mid [q_1q_2^{-1},q_1^{-1}q_2q_1], [q_1q_2^{-1},q_1^{-2}q_2q_1^2]\rangle\cong F.$$
In particular, the subgroup of $ \langle p_1,p_2\rangle\times \langle q_1,q_2\rangle$
generated by the elements $(p_1,q_2)$ and $(p_2,q_1)$ is isomorphic to $P$.

Now we claim that the map $$h_2\mapsto \lambda_2\lambda_1^{-1},\qquad h_1\mapsto \lambda_1^{-1}\lambda_2,$$
extends to an embedding
$$\langle h_1,h_2 \rangle\to \langle \lambda_1,\lambda_2\rangle.$$
This is true for the component $J_{10}$, where this is a dynamical realization of $H$.
So it suffices to show that each relation in $h_1$ and $h_2$ is satisfied by the restrictions of $\lambda_1^{-1}\lambda_2$ and $\lambda_2\lambda_1^{-1}$
on other components.
As we saw before, for $J_6$, this via the $\mathbb{Z}$--quotient given by killing the normal closure of the generator $h_1\in H$,
which factors through the abelianization map.
For the components $$[0,1]\bigcup J_{\{1,\ldots,5\}},\qquad J_{\{12,\ldots,16\}}\bigcup [2,3],$$
the action of $H$ is precisely as $P$, and since $H$ is a marked extension of $P$, whence the desired conclusion.  
\end{proof}

\subsection{Smoothability}\label{sec:smooth}

To construct square roots of $F$ which are not conjugate into $\Diff^2(I)$ or $\Diff^2(\R)$, the group of $C^2$ orientation-preserving diffeomorphisms of the interval and the real line respectively, we use the following result due to Plante--Thurston and its generalizations due to Farb--Franks:

\begin{thm}[See~\cite{PT1976,FF2003}]\label{thm:PT}
Let $N<\Diff^2(M)$ be a finitely generated nilpotent subgroup, where here $M$ is a compact and connected one--manifold. Then $N$ is abelian. Moreover, any nilpotent subgroup of $\Diff^2(\R)$ is metabelian.
\end{thm}

\begin{proof}[Proof of Corollary~\ref{cor:smooth}]
Let $N_k=\bf F_2/\mathcal{R}^k$ be as in Lemma \ref{lem:solvable}. 
Then $\mathcal{R}/\mathcal{R}^k<N_k$. 
Taking a finite subset of a free generating set $S$ for $\mathcal{R}$, we have that the image of $S$ in $N_k$ generates a nilpotent subgroup 
$\Gamma_S<N_k$. 
It is straightforward to check that $\Gamma_S$ is a retract of $\mathcal{R}/\mathcal{R}^k$, 
and is therefore a nonabelian nilpotent subgroup of $N_k$ whenever $k\geq 3$. 
Applying Theorem \ref{thm:main} and Theorem \ref{thm:PT} gives the desired conclusion in the case where $M$ is compact. Choosing a $k\gg 0$ such that $N_k$ contains a nilpotent subgroup which is not metabelian, we get the desired conclusion for $\R$ as well.
\end{proof}

Corollary~\ref{cor:pl} similarly follows from Theorem~\ref{thm:main} and Theorem 4.1 of~\cite{FF2003}.

\section{Uncountability of $\mathcal{S}$ and infinitely presented examples}

In this section, we prove Theorem \ref{thm:uncountable}. For this, we retain the notation from the previous discussion.

\subsection{Sources of uncountability}

A construction of P. Hall (sometimes attributed to B. Neumann) on the existence of uncountably many distinct isomorphism classes of two--generated groups as outlined by de la Harpe in part III.C.40 of~\cite{MR1786869} has the advantage that the resulting groups are all left orderable, as observed in~\cite{KKL16}. We summarize the relevant conclusions here:

\begin{prop}\label{prop:neumann}
There exists an uncountable class $\mathcal{N}$ of pairwise non--isomorphic groups such that if $N\in\mathcal{N}$ then $N$ is two--generated, left orderable, and $N^{ab}=\Z^2$. In particular, $N$ can be realized as a subgroup of $\Homeo^+(\R)$.
\end{prop}

The reader will also find groups in the class $\mathcal{N}$ described explicitly below in the proof of Corollary~\ref{cor:uncgeneral}.

\subsection{Equations over $\Homeo^+(\R)$}

In order to prove Theorem~\ref{thm:uncountable}, we will construct an explicit orderable marked extension of $P$ which contains a given element of $\mathcal{N}$ as a subgroup. To do this, we will need to solve equations over $\Homeo^+(\R)$.

Let $\{f_1,\ldots,f_k,g\}\subset\Homeo^+(\R)$ be given, and let $w\in \bf F_n$ be a reduced word in the free group on $n$ fixed generators, where here $k<n$. An \emph{equation} over $\Homeo^+(\R)$ is an expression of the form 
\[w(f_1,\ldots,f_k,x_1,\ldots,x_{n-k})=g.\] A tuple $\{y_1,\ldots,y_{n-k}\}\subset\Homeo^+(\R)$ is a \emph{solution} to the equation if this expression becomes an equality after substituting $y_i$ for $x_i$ for each $i$, and interpreting the expression in $\Homeo^+(\R)$.

We will restrict out attention to the case where $n=2$. Even here, equations may not admit solutions. A trivial example can be given by taking $f\neq g$ and setting $w$ to be the first free generator. A slightly less trivial example can be given by taking $f$ to be fixed point free, taking $g$ to have at least one fixed point, and setting $w$ to be a conjugate of the first free generator.

We will concern ourselves with a particular commutator word $w$ with free generators $s$ and $t$, so that under the map ${\bf F_2}\to P$ given by $s\mapsto (p_1,q_2)$ and $t\mapsto (p_2,q_1)$, 
the element $w$ lies in the kernel. 

%Indeed, under the natural embedding $P\to F\times F$, the commutator $[st^{-1},s^{-2}ts^2]$ is the identity in the left factor and the commutator $[st^{-1},t^{-1}s^{-1}t]$ is the identity on the right factor.

The following lemma is key in proving Theorem~\ref{thm:uncountable}:

\begin{lem}\label{lem:equation}
Fix a group $N\in \mathcal{N}$ and let $\tau$ be the map $\tau(t)=t+1$. 
There is a homeomorphism $\kappa\in \Homeo^+(I)$ and a nontrivial commutator word $w\in \ker\{{\bf F_2}\to P\}$ such that:
\begin{enumerate}
\item The group $\langle \kappa,\tau \rangle$ contains $N$ as a subgroup.
\item The equation $w(\tau,x)=\kappa$ admits a solution $y\in\Homeo^+(\R)$.
\end{enumerate}
\end{lem}

We first show how Lemma~\ref{lem:equation} implies Theorem~\ref{thm:uncountable}:

\begin{proof}[Proof of Theorem~\ref{thm:uncountable}]
We recall some of the notation and the construction in Subsection~\ref{subsec:building}.
We will use informal language below, since we already have a precise description in that subsection.

Given any $h_1,h_2\in \Homeo^+(I)$, we can build a square root $G\in\mathcal{S}$ generated by $\lambda_1,\lambda_2$ such that the group $H=\langle \lambda_1^{-1}\lambda_2,\lambda_2\lambda_1^{-1}\rangle$ satisfies the following.
\begin{enumerate}
\item $H$ acts as a dynamical realization of $P$ on $$([0,1]\cup J_{\{1,...,5\}})\bigcup (J_{\{12,...,16\}}\cup [2,3])$$
\item The group $\langle \lambda_2\lambda_1^{-1}\rangle$ acts faithfully by $\mathbb{Z}$ on the interval $J_6$ and the element $ \lambda_1^{-1}\lambda_2$ acts trivially on the interval $J_6$.
\item The element $\lambda_1^{-1}\lambda_2$ acts as $h_1$ on $J_{10}$ and the element $\lambda_2\lambda_1^{-1}$ acts as $h_2$ on $J_{10}$.
\item The action of $H$ outside the above intervals is trivial.
\end{enumerate} 

Let $\tau,\kappa$ and $y$ be the homeomorphisms of the real line from Lemma \ref{lem:equation}.
For the rest of the proof, we fix dynamical realizations of $\tau,\kappa$ on $J_{10}$
obtained from conjugating by a homeomorphism of $\mathbb{R}$ to the interior of $J_{10}$.
We shall now denote by $\tau,\kappa,y$ as these homeomorphisms supported on $J_{10}$.

We use the input $h_1=\tau$ and let $h_2=y$ to produce a square root $G$ of $F$.
Consider the subgroup $K$ of $H$ generated by 
$$k_1=\lambda_1^{-1}\lambda_2,\qquad k_2=w(\lambda_1^{-1}\lambda_2,\lambda_2\lambda_1^{-1}).$$
We check the following:
\begin{enumerate}
\item $k_1\restriction J_{10}=\tau$ and $k_2\restriction J_{10}=\kappa$.
\item $k_2$ acts trivially outside $J_{10}$
since $w(s,t)$ represents the identity in $P$ under the map $s\mapsto (p_1,q_2)$ and $t\mapsto (p_2,q_1)$.
\item Any commutator vanishes on $J_6$.
\item $k_1$ acts trivially on $J_6$ and by $\mathbb{Z}$ on $$([0,1]\cup J_{\{1,...,5\}})\bigcup (J_{\{12,...,16\}}\cup [2,3])$$
\end{enumerate}

By our assumption, $N<\langle k_1,k_2\rangle\restriction J_{10}$.
We claim that in fact, $N<\langle k_1,k_2\rangle$. 
This follows from the fact that the relations in $N$ are elements of the commutator subgroup of the free group,
and since $\langle k_1, k_2\rangle$ acts by $\mathbb{Z}$ outside $J_{10}$.
Therefore $N<G$ where $G$ is the corresponding square root of $F$.

We thus obtain that if $N\in\mathcal{N}$ is given, then there is a square root $G_N\in\mathcal{S}$ which contains $N$ as a subgroup. 
Since the class $\mathcal{N}$ contains uncountably many different isomorphism types and since any element of $\mathcal{S}$ is two--generated and hence countable, 
the class $\{G_N\mid N\in\mathcal{N}\}\subset\mathcal{S}$ consists of uncountably many different isomorphism types.
\end{proof}

\begin{proof}[Proof of Lemma~\ref{lem:equation}]
We shall use the commutator word $$w(s,t)=[w_1(s,t), w_2(s,t)],$$
where $$w_1=[st^{-1},s^{-2} t s^2],\qquad w_2=t [st^{-1}, t^{-1} s^{-1} t] t^{-1}.$$
It is straightforward to check that for the map ${\bf F_2}\to P$ given by $s\mapsto (p_1,q_2)$ and $t\mapsto (p_2,q_1)$, 
the element $w$ lies in the kernel. 

Let $\phi,\psi\in\Homeo^+(I)$ be given generators of $N$. 
We first choose homeomorphisms $\mu,\nu\in\Homeo^+(I)$ such that $\psi=[\mu,\nu^{-1}]$ and homeomorphisms 
$\chi,\xi\in\Homeo^+(I)$ such that $\phi=[\xi,\chi^{-1}]$. Such choices are possible, since every element of $\Homeo^+(I)$ is a commutator (see Theorem 2.65 of~\cite{Calegari2007}, for instance).

Identify $I$ with the unit interval $[0,1]\subset\R$. 
Recall that $\tau$ is translation by $1$ on $\R$.
We set $\kappa\in\Homeo^+(\R)$ as 
$$\kappa=(\tau^{-2}\psi\tau^2) (\tau^{-102}\phi\tau^{102})$$
Intuitively, $\kappa$ acts by $\psi$ on the interval $[2,3]$, by $\phi$ on the interval $[102,103]$, and by the identity otherwise.
We now verify that $\kappa$ witnesses the conditions of the lemma. 

We set \[y=(\tau^{-1}\mu\tau^{1})(\tau^{-2}\nu\tau^{2}) (\tau^{-101}\chi\tau^{101}) (\tau^{-102}\xi\tau^{102}).\] 
Intuitively, the homeomorphism $y$ acts by $\mu$ on $[1,2]$, by $\nu$ on $[2,3]$, by $\chi$ on $[101,102]$, and by $\xi$ on $[102,103]$. 
We check that $y$ is the solution to the equation.
We proceed by analysing the two inner commutators separately, and then considering the outer commutator. 

Consider the commutator $[\tau y^{-1},\tau^{-2}y\tau^2]$.
First note that since $y^{-1}$ has disjoint support from $\tau^{-2} y \tau^2$, they commute, and hence 
$$[\tau y^{-1},\tau^{-2}y\tau^2]=[\tau,\tau^{-2}y\tau^2]$$
We can now easily check that the action of the resulting homeomorphism is as follows: 
\begin{enumerate}
\item It acts by $\mu$ on $[2,3]$, by $\nu\mu^{-1}$ on $[3,4]$, and by $\nu^{-1}$ on $[4,5]$.
\item It acts by $\chi$ on $[102,103]$, by $\xi\chi^{-1}$ on $[103,104]$, and by $\xi^{-1}$ on $[104,105]$. 
\end{enumerate}
We denote this homeomorphism by $\alpha$.

Next, consider the commutator $$[\tau y^{-1}, y^{-1} \tau^{-1} y]$$
First, note that $$[\tau y^{-1},y^{-1}\tau^{-1}y]=[\tau,y^{-2}][\tau^{-1},y^{-1}]$$

It is straightforward to check that the homeomorphism resulting from the product of these commutators is as follows:
\begin{enumerate}
\item It acts by $\mu^{-2}$ on $[0,1]$, by $\nu^{-2}\mu^3$ on $[1,2]$, by $\nu^2\mu^{-1}\nu$ on $[2,3]$, and by $\nu^{-1}$ on $[3,4]$. 
\item  It acts by $\chi^{-2}$ on $[100,101]$, by $\xi^{-2}\chi^3$ on $[101,102]$, by $\xi^2\chi^{-1}\xi$ on $[102,103]$, and by $\xi^{-1}$ on $[103,104]$.
\end{enumerate}
We denote by $\beta$ the homeomorphism $t [\tau y^{-1}, y^{-1} \tau^{-1} y] t^{-1}$.

Finally, we consider the homeomorphism $[\alpha,\beta]$. 
Observe that the supports of $\alpha$ and $\beta$ intersect in the intervals $[2,3]$ and $[102,103]$. 
Since $\psi=[\mu,\nu^{-1}]$, 
we see that $[\alpha,\beta]$ acts by $\psi$ on $[2,3]$. 
Similarly, since $\phi=[\chi,\xi^{-1}]$, we have that $[\alpha,\beta]$ acts by $\phi$ on $[102,103]$. 
It follows that $[\alpha,\beta]$ agrees with $\kappa$, whence $y$ is a solution to the equation as claimed.

Finally, we show that $N<\langle \kappa, \tau\rangle$.
Indeed the group generated by $\tau^{-100}\kappa \tau^{100}, \kappa$ acts as $N$ on $[102,103]$ and as $\mathbb{Z}$ outside this interval.
Since the relations in $N$ are elements of the commutator subgroup of the free group, it follows that this group is isomorphic to $N$.
\end{proof}

\section{General square root phenomena}\label{sec:generalroot}
In this section, we pass to the completely general setup of finitely generated subgroups of $\Homeo^+(I)$ and address the results in Subsection~\ref{subsec:other}.

\subsection{Roots of homeomorphisms}

We begin with a completely general construction in $\Homeo^+(I)$ for producing roots of homeomorphisms.
The following is a well--known fact, whose proof we recall for the convenience of the reader:

\begin{lem}\label{lem:root}
Let $f\in\Homeo^+(I)$. Then for all $n\in\bN$, there exists an element $g=g_n\in\Homeo^+(I)$ such that $g^n=f$.
Moreover, there are uncountably many possible choices of such a map $g$.
\end{lem}

\begin{proof}
By considering the components of the support of $f$ separately, we may consider the case where $f$ has no fixed points in the interval $(0,1)$. In this case, $f$ is topologically conjugate to the homeomorphism of $\R\cup\{\pm\infty\}$ given by $x\mapsto x+1$.

We now build an $n^{th}$ root of $f$ defined on all of $\R$ in the following manner. 
First, we choose arbitrary orientation-preserving homeomorphisms 
\[h_m:[\frac{m}{n},\frac{m+1}{n}]\to [\frac{m+1}{n},\frac{m+2}{n}]\]  for $0\leq m\leq n-2$.
Next, we inductively define homeomorphisms \[h_m:[\frac{m}{n},\frac{m+1}{n}]\to [\frac{m+1}{n},\frac{m+2}{n}]\] for all $m\in \Z$, 
such that 
\[h_{k+(n-1)}\circ \cdots \circ h_{k}=x+1\]
for each $k\in \mathbb{Z}$. 
It is clear then that the homeomorphisms $h_m$ piece together to give a homeomorphism $g$ of $\R\cup\{\pm\infty\}$, whose $n^{th}$ power is translation by one. Moreover, the arbitrariness of the choices made guarantees that there are uncountably many choices for $g$.
\end{proof}

\subsection{Free groups}

Classical result from combinatorial and geometric group theory show that there exist two--generated groups $\langle a,b\rangle$ which are not free, but such that $\langle a^2,b^2\rangle$ is free. Moreover, one can arrange for these groups to be left orderable, and hence to be realized as subgroups of $\Homeo^+(I)$. For instance, we take the braid group on three strands \[B_3=\langle a,b\mid aba=bab\rangle.\] All braid groups are left orderable (in fact the reader may consult~\cite{DDRW08} as a book dedicated entirely to this subject),
and it is a standard fact that the squares of the standard braids generate a free group (see Chapter 3, Section 5 of~\cite{FM2012}).

\subsection{The lamplighter group}

Using square roots of $F$, we can produce many square roots of the lamplighter group $L=\Z\wr\Z$. Recall that \[\Z\wr\Z\cong\Z\ltimes \big(\bigoplus_{i\in\Z}\Z_i\big),\] where the natural action of $\Z$ is by translating the index $\Z_i\mapsto\Z_{i+1}$. The group $L$ is naturally realized as a subgroup of $\Homeo^+(\R)<\Homeo^+(I)$ as follows. We choose an arbitrary homeomorphism $\psi$ such that $\supp\psi=(0,1)\subset\R$, and then we consider the group generated by $\psi$ and $\tau$, where as before $\tau(x)=x+1$. It is clear that $\langle\psi,\tau\rangle\cong L$. The following result clearly implies Theorem~\ref{thm:lamplighter}, in light of Theorem~\ref{thm:uncountable}

\begin{thm}\label{thm:lamproot}
Let $G$ be a left orderable marked extension of $P$. Then there exists a square root of $L$ containing an isomorphic copy of $G$.
\end{thm}
\begin{proof}[Sketch of proof]
Let $T\in\Homeo^+(\R)$ be given by $T(x)=x+1/2$. Note that the intervals $(0,1)$ and $T((0,1))$ together form a chain of intervals.

Let $G$ be a given left orderable marked extension of $P$, with distinguished generators $g_1$ and $g_2$.
Let $\psi$ be a homeomorphism supported on $(0,1)$ satsifying the following
conditions:

\begin{enumerate}
\item
$\psi(1/2)=1/2$.
\item
The group $\langle\psi,T\psi T^{-1}\rangle<\Homeo^+(\R)$ contains a copy of $G$.
\end{enumerate}

It is easy to see that such a $\psi$ exists, since besides the two conditions above it is otherwise arbitrary,
and its action on the two halves of $(0,1)$
can be chosen independently. Thus, one may arrange 
that the action on $(0,1/2)$ recovers the generator $g_1$ and the action on $(1/2,1)$ recovers the generator
$g_2$. It is clear that squaring $\psi$ and $T$ gives rise to a group isomorphic to the lamplighter group. If one would like the copy of $G$
to lie in a square root of $F$, one can further compose $\psi$ with a suitably chosen increasing homeomorphism of $(0,1)$.
The reader should compare with Subsection~\ref{subsec:building} for the details of the latter construction.
\end{proof}

\subsection{Square roots of $\Z$}
In this section, we give a recipe for producing many $n$--generated groups of homeomorphisms of the interval, so that the squares of the generators generate a cyclic group.

\begin{proof}[Proof of Theorem~\ref{thm:Z^n}]
Let $\tau_1,\ldots,\tau_{n+1}$ be $n+1$ copies of the translation $\tau(x)=x+1$ viewed as a homeomorphism of $\R$, and let \[\langle h_1,\ldots,h_n\rangle=H<\Homeo^+(I)\] be an arbitrary $n$--generated subgroup. We set $T_{n+1}(x)=x+1/2$.

Lemma~\ref{lem:root} constructs all possible roots $\tau$, and we follow the construction given there. We first scale down $H$ to be a group of homeomorphisms of $[0,1/2]$, and we abuse notation and label the generators of $H$ by $\{h_1,\ldots,h_n\}$. We now define $T_i$ to be the homeomorphism $T_{n+1}\circ h_i$ on $[0,1/2]$. The requirement $T_i^2=\tau$ determines the values of $T_i$ on the rest of $\R$.

Now let $S_i=T_{n+1}^{-1}\circ T_i$ for $1\leq i\leq n$. Observe that $S_i$ acts by $h_i$ on each interval of the form $[k,k+1/2]$ and by $h_i^{-1}$ on each interval of the form $[k-1/2,k]$, where $k\in\Z$. It is clear then the $\yh H\cong \langle S_1,\ldots,S_n\rangle<\langle T_1,\ldots,T_{n+1}\rangle$.
\end{proof}

\begin{proof}[Proof of Corollary~\ref{cor:uncgeneral}]
First, note that if $H$ is free then the skew subdirect product $\yh H$ is also free, which establishes the second part of the corollary. For the first part, we perform a mild modification of the Hall--Neumann groups as discussed in Proposition~\ref{prop:neumann} (see~\cite{KKL16} for a detailed discussion of these groups). All the Hall--Neumann groups quotients of a single two--generated group $\Gamma=\langle t,s_0\rangle$ which is left orderable. The element of $\mathcal{N}$ are given as quotients of $\Gamma$ by certain central normal subgroups $N_X<\Gamma$.

We recall the definition of $\Gamma$ for the convenience of the reader, following III.C.40 of de la Harpe's book~\cite{MR1786869}
(cf. Lemma 5.1 of~\cite{KKL16}). We begin
with a set $S=\{s_i\}_{i\in\Z}$. Then, we define \[R=\{[[s_i,s_j],s_k]=\}_{i,j,k\in\Z}\cup \{[s_i,s_j]=[s_{i+k},s_{j+k}]\}_{i,j,k\in\Z}.\] The group
$\Gamma_0$ is defined by $\langle S\mid R\rangle$, and $\Gamma$ is defined as a semidirect product of $\Z$ with $\Gamma_0$, where
the generator $t$ of the $\Z$--factor acts by $t^{-1}s_it=s_{i+1}$. One sets $u_i=[s_0,s_i]$, and if $X\subset \Z\setminus \{0\}$, we write
$N_X=\langle \{u_i\}_{i\in X}\rangle$.

Note that the group $\Gamma$ is generated by $t$ and $s_0$.
It is straightforward to check that the map given by $t\mapsto t^{-1}$ and $s_0\mapsto s_0^{-1}$ extends to a well--defined automorphism of $\Gamma$, whence $\yh\Gamma\cong\Gamma$. Moreover, the subgroups $N_X$ are all stable under this automorphism of $\Gamma$. In particular, it follows that if $N\in\mathcal{N}$ is one of the Hall--Neumann groups, then $N\cong\yh N$. The first claim of the corollary follows from Theorem~\ref{thm:Z^n}.
\end{proof}

\subsection{General groups}\label{subsec:generalroot}
For a general finitely generated group $H=\langle x_1,\ldots,x_n\mid R\rangle$, one can \emph{formally take the square root} of $H$ by setting \[G=\langle y_1,\ldots,y_n,x_1,\ldots,x_n\mid R, x_1=y_1^2,\ldots,x_n=y_n^2\rangle.\] Note that this definition depends on the presentation of $H$ which is given. If $H$ is given as a free group with no relations then $G$ will be free of the same rank. However, if $H$ is not freely presented then $G$ can be very complicated. We will call a presentation for a group $H$ \emph{reduced} if $x_i$ is nontrivial in $H$ for each $i$.

\begin{thm}\label{thm:formalorder}
Let $H$ be a left orderable finitely generated group with a reduced presentation. Then the formal square root $G$ of $H$ is left orderable.
\end{thm}
\begin{proof}
If $H=\langle x_1,\ldots,x_n\mid R\rangle$, set \[K=\langle y,x_1,\ldots,x_n\mid R,x_1=y^2\rangle.\] If we can prove that $K$ is left orderable then the result will follow by induction on $n$.

To this end, note that $K$ admits a description as an amalgamated product via \[K=\Z*_{2\Z=\langle x_1\rangle}H.\] Since we can order $\Z$ either positively or negatively, we may assume that the isomorphism $2\Z\cong \langle x_1\rangle$ is order preserving. Then, a result of Bludov--Glass~\cite{BludovGlass} (cf. Bergman~\cite{Bergman90}) implies that the corresponding amalgamated product is again orderable.
\end{proof}

Note that the assumption that the presentation for $H$ in Theorem~\ref{thm:formalorder} is reduced was essential, since otherwise the formal square root would contain torsion. Moreover, we need not assume that $H$ be finitely generated in Theorem~\ref{thm:formalorder}, and this hypothesis could be replaced by countable generation. We include this hypothesis since nearly all groups under consideration in this paper are finitely generated.

Finally, we show that formal square roots generally contain nonabelian free groups, so that free subgroups are in some precise sense a general phenomenon in square roots of groups of homeomorphisms:

\begin{thm}\label{thm:generalfree}
Let $H=\langle x_1,\ldots,x_n\mid R\rangle$ be a reduced presentation for a non--cyclic finitely generated left orderable group, and let $K=\langle y,x_1,\ldots,x_n\mid R,x_1=y^2\rangle$. Then $K$ contains a nonabelian free group.
\end{thm}

Thus, Theorem~\ref{thm:generalfree} implies that the formal square root of a non--cyclic group always contains nonabelian free groups.

\begin{proof}[Proof of Theorem~\ref{thm:generalfree}]
The result follows from general Bass--Serre theory. One can construct free subgroups explicitly using the standard theory of
 normal forms for amalgamated products
 (see~\cite{Trees,LS2001} for general introductions to combinatorial group theory and in particular Theorem 2.6 in \cite{LS2001} for the normal form theorem for amalgamated products). 
 To do this, let $z\in H\setminus\langle x_1\rangle$, which exists since $H$ is assumed not to be cyclic. Note that $z$ has infinite order since $H$ is left orderable. Consider the group $\langle z,yzy^{-1}\rangle$. An arbitrary word in these generators will be of the form \[z^{n_1}yz^{m_1}y^{-1}\cdots z^{n_k}yz^{m_k}y^{-1},\] where all these exponents are nonzero except possibly $n_1$ and $m_k$. This word cannot collapse to the identity since it is in normal form. It follows that the group $\langle z,yzy^{-1}\rangle$ is free.
\end{proof}

M. Kassabov has pointed out to the authors that if $H$ is given a presentation with at least three generators then the formal square root $G$ of $H$ surjects onto \[\Z/2\Z*\Z/2\Z*\Z/2\Z,\] which contains a nonabelian free group.

Finally, we remark that there appears to be little general compatibility between formal square roots of groups and ``dynamical" square roots of groups, such as in our discussion of square roots of $F$ in this paper. That is, let $G=\langle f,g\rangle\in\mathcal{S}$ be a square root of $F$, so that the supports of $f$ and $g$ form a two--chain. Then it is never the case that $G$ is the formal square root of $F$. Indeed, by Theorem~\ref{thm:generalfree}, we would have that $g$ and $h=f^{-1}gf$ generate a free group. This cannot happen, since there is an endpoint $x$ of $\supp g$ which is fixed by both $g$ and $h$, and where the germs of these two homeomorphisms at $x$ agree. By conjugating $g$ or $h$ suitably, one can obtain a homeomorphism $k$ such that $\supp k$ is contained in a neighborhood of $x$ on which $g$ and $h$ agree. But then $g^{-1}kg=h^{-1}kh$, violating the fact that $\langle g,h\rangle$ is free.

If on the other hand $f$ and $g$ are fully supported homeomorphisms, then fairly easy Baire Category methods as in Proposition 4.5 of~\cite{Ghys2001} show that by choosing generic square roots of $f$ and $g$, one obtains a group which is isomorphic to the formal square root of $\langle f,g\rangle$.

%%%%%%%%%%%%%%%%%%%%%%%%%%%%%%%%%%%%%%%%%%%%%%%%%%%%%%%%%%%%%%%

%%%%%%%%%%%%%%%%%%%%%%%%%%%%%%%%%%%%%%%%%%%%%%%%%%%%%%%%%%%%%%%

%%%%%%%%%%%%%%%%%%%%%%%%%%%%%%%%%%%%%%%%%%%%%%%%%%%%%%%%%%%%%%

\section*{Acknowledgements}
The authors thank C. Bleak, M. Kassabov, J. Moore, and D. Osin for helpful discussions. The authors thank M. Brin for pointing out Corollary~\ref{cor:pl}. The authors are particularly grateful to M. Sapir for suggesting the content of Section~\ref{sec:generalroot}. The authors thank an anonymous referee for providing a large number of helpful comments. The first author was partially supported by Simons Foundation Collaboration Grant number 429836, and is partially supported by an Alfred P. Sloan Foundation Research Fellowship and by NSF Grant DMS-1711488. The second author has been supported by an EPFL-Marie Curie fellowship and the Swiss National Science Foundation Grant ``Ambizione" PZ00P2\_174137.

%%%%%%%%%%%%%%%%%%%%%%%%%%%
% END of body
%%%%%%%%%%%%%%%%%%%%%%%%%%%

\bibliographystyle{amsplain}

\end{document}